\documentclass[12pt]{amsart}

\usepackage{amssymb}

\usepackage{enumitem}

\usepackage{graphicx}

\makeatletter
\@namedef{subjclassname@2020}{%
  \textup{2020} Mathematics Subject Classification}
\makeatother

\usepackage[T1]{fontenc}

\newtheorem{theorem}{Theorem}[section]

\newtheorem{lemma}[theorem]{Lemma}

\theoremstyle{definition}
\newtheorem{definition}[theorem]{Definition}
\newtheorem{remark}[theorem]{Remark}

\numberwithin{equation}{section}

\begin{document}

\baselineskip=17pt

\title[A divisor problem for polynomials]{A divisor problem for polynomials}
\author[B. Klahn]{Benjamin Klahn}

\address{Institute of Analysis and Number Theory \\ Technical University of Graz\\
\'Kopernikusgasse 24/II\\
8010 Graz, Austria}
\email{klahn@math.tugraz.at}

\date{}

\begin{abstract}
    We characterize all monic polynomials $f(x) \in \mathbb{Z}[x]$ that have the property that
\[f(p) \mid f(p^{p}),~\text{for all sufficiently large primes }p \geq N(f). \]
We also give necessary conditions and a sufficient condition for monic polynomials $f(x) \in \mathbb{Z}[x]$ to satisfy $f(p) \mid f(p^{p})$ for all primes $p$.
\end{abstract} 


\maketitle

\section{Introduction} For integer polynomials $f(x)$ it is a very difficult problem to say something useful about the size and the number of prime divisors of values $f(n)$, $n \in \mathbb{Z}$. It is conjectured that irreducible integer polynomials $f(x)$ without any fixed prime divisor take prime values for infinitely many $n$, and it is even believed that there are infinitely many prime numbers $p$ for which $f(p)$ is prime. This is a well-studied problem and there are deep results towards this conjecture. In \cite{ID82} Deshouillers and Iwaniec proved that there are infinitely many $n$ such that $n
^{2}+1$ has a prime divisor larger than $n^{6/5}$. In \cite{HB00} Heath-Brown proved that for an explicitly given positive $\delta$ the largest prime divisor of $n^{3}+2$ is larger than $X^{1+\delta}$ for a positive proportion of integers $n$ in the interval $[X,2X]$. Later Irving improved this result in \cite{I15} by showing that one can take $\delta = 10^{-52}$. However, it is not known for a single primitive irreducible polynomial of degree at least 2 whether it takes infinitely many prime values.

On the other hand, in \cite{BFMW20} it was shown by Bober, Fretwell, Martin and Wooley that for a fixed quadratic integer polynomial $f(x)$ and a positive $\epsilon$ there are infinitely many integers $n$ such that all prime divisors of $f(n)$ are less than $n^{\epsilon}$.

Thus, it is typically a very difficult problem obtaining information about the prime divsiors of elements of a sequence $(f(n))_{n=m}^{M}$, and initially it therefore came as a surprise to the author that it was possible to relate the prime divsisors of $f(p)$ and $f(p^{p})$.

However, the density theorem of Chebotarëv often comes in very handy as it gives complete information on the density of primes which divide some value of $f(x)$ in terms of its Galois group $\text{Gal}(\text{split}(f(x),\mathbb{Q})/\mathbb{Q})$ where $\text{split}(f(x),\mathbb{Q})$ is the splitting field of $f(x)$. In fact, Chebotarëv's density theorem is one of the key tools in our study of which monic integer polynomials, $f(x)$, satisfy the condition
\begin{equation}
    f(p) \mid f(p^{p})~\text{for all sufficiently large primes}~p \geq N(f). \label{theProblem}
\end{equation}
Related to the problem of determining the polynomials that satisfy (\ref{theProblem}) is the following problem.
\\
~\\
\textit{
    Find all integer polynomials $f(x)$ with nonnegative coefficients such that
    \[\text{rad}(f(n)) \mid f(n^{\text{rad}(n)})\]
    for all nonnegative integers $n$.
}
\\
~\\
Here the \textit{radical}, $\text{rad}(\cdot)$, is defined in the following way.
\begin{definition}
For a nonnegative integer $n$ define $\text{rad}(n) = 1$ if $n = 0$ or $n = 1$, and otherwise $\text{rad}(n) = p_{1}p_{2} \cdots p_{k}$ where $p_{1} < p_{2} < \cdots < p_{k}$ are all the distinct prime factors of $n$.
\end{definition}

The problem above was submitted for the International Mathematical Olympiad in 2012\footnote{A polynomial, $f(x)$, satisfies this condition if and only if $f(x) = ax^{m}$ for some nonnegative integers $a$ and $m$}, see \cite{IMO12}. Altering the formulation of the olympiad problem to only having information at primes makes the problem harder, and the answer to the problem of which monic integer polynomials satisfy (\ref{theProblem}) is richer in the sense that there are more types of polynomials satisfying this condition than the condition of the olympiad problem.

\section{Results}
The main result in this paper is Theorem \ref{main1} which gives a necessary and sufficient condition for (\ref{theProblem}) to be satisfied.
\begin{theorem} \label{main1}
Let $f(x) \in \mathbb{Z}[x]$ be a monic polynomial. There is a positive integer $N:=N(f)$ such that
\begin{equation}
    f(p) \mid f(p^{p}),~\text{for all primes } p \geq N, \label{divisorprop}
\end{equation}
if and only if there are positive integers $d_{1},...,d_{m}$ and nonnegative integers $e_{0},e_{1},...,e_{m}$ such that
\begin{equation}
    f(x)=x^{e_{0}}\Phi_{d_{1}}(x)^{e_{1}}\Phi_{d_{2}}(x)^{e_{2}}\cdots \Phi_{d_{m}}(x)^{e_{m}}, \label{primeFac}
\end{equation}
where $\Phi_{d}(x)$ is the $d$-th cyclotomic polynomial.
\end{theorem}
Typically a polynomial of the form (\ref{primeFac}) will not satisfy (\ref{divisorprop}) with $N=2$. Of course there are the trivial solutions $f(x) = x^{m}$ for an arbitrary positive integer $m$, that satisfy (\ref{divisorprop}) with $N=2$. Furthermore, all polynomials of the form $f(x) = x^{m}-1$ for an arbitrary positive integer $m$, satisfy (\ref{divisorprop}) with $N=2$. A less trivial example of such a polynomial is $f(x) = (x
-1)(x+1)(x^{2}+x+1)$. It turns out that a polynomial that satisfies (\ref{divisorprop}) with $N=2$ and is not a power of $x$ has 1 as a root.
\begin{theorem} \label{oneRoot}
Let $f(x) \in \mathbb{Z}[x]$ be a monic polynomial that satisfies (\ref{divisorprop}) with $N=2$. If $f(x)$ is not a power of $x$ then $x-1 = \Phi_{1}(x) \mid f(x)$. Furthermore, a polynomial of the form (\ref{primeFac}) with $d_{1} = 1$ and $d_{j} = p_{j}$ for $j \geq 2$ for distinct primes $p_{j}$ satisfies (\ref{divisorprop}) with $N=2$.
\end{theorem}
\section{Preliminaries}
Along the way to deriving Theorem \ref{main1} we will make use of the following consequence of the Chebotarëv density theorem.
\begin{theorem}[\textbf{\cite{SL93}}] \label{specialChebotarev}
Let $f(x) \in \mathbb{Z}[x]$ be a monic polynomial having nonzero discriminant $\Delta(f(x)) \neq 0$. For a positive proportion of the primes, the reduction $f(x) \mod p \in \mathbb{F}_{p}[x]$ will split completely as the product of distinct linear factors.
\end{theorem}
Theorem \ref{specialChebotarev} will be crucial because it will allow us to apply the following trick that later will show that any solution $f(x)$ must, essentially, be a product of cyclotomic polynomomials.
\begin{lemma} \label{rootimplication}
Let $g(x)$ and $h(x)$ be monic integer polynomials with $h(x)$ having discriminant $\Delta(h(x)) \neq 0$. Assume that for infinitely many primes $p$, for which $h(x) \mod p$ splits completely, $g(x)$ has the following property:
\begin{equation}
    k \in \mathbb{Z}~\text{and } h(k) \equiv 0 \pmod{p} \Longrightarrow g(k) \equiv 0 \pmod{p}. \label{rootImplication}
\end{equation}
Then $h(x)$ divides $g(x)$ in $\mathbb{Z}[x]$. 
\end{lemma}

\begin{remark}
In relation to Lemma \ref{rootimplication} we note that for monic integer polynomials $g(x)$ and $h(x)$ satisfying the conditions in Lemma \ref{rootimplication} it holds that for infinitely many primes $p$, $(h(x) \mod p) \mid (g(x) \mod p)$. As such, Lemma \ref{rootimplication} can be generalized since the existence of infinitely many primes $p$ such that $(h(x) \mod p) \mid (g(x) \mod p)$, in fact, implies that $h(x)$ divides $g(x)$ in $\mathbb{Z}[x]$ -- regardless of $h(x)$ having nonzero discriminant. However, for later convienience we prefer to use the formulation of Lemma \ref{rootimplication} as we believe it will make the proof of Theorem \ref{main1} more clear.
\end{remark}

We give the proof of Lemma \ref{rootimplication} in the next section.

In order to show that all polynomials of the form (\ref{primeFac}) satisfy (\ref{divisorprop}) for all large primes we will use the following two elementary properties of cyclotomic polynomials.
\begin{lemma} \label{sufLemma}
For cyclotomic polynomials the following holds:
\begin{enumerate}
\item[(1)] For a prime $p$ and a positive integer $d$, we have
\[
\Phi_{d}\left( x^{p} \right) =
\begin{cases}
\Phi_{pd}(x) &\text{if } p \mid d \\ \Phi_{pd}(x) \Phi_{d}(x) &\text{if } p \nmid d. \end{cases}
\]
\item[(2)] If $p$ is a prime and $n$ and $m$ are positive integers such that
\[
    p \mid (\Phi_{n}(b),\Phi_{m}(b))
\]
for some integer $b \geq 2$, then $m/n = p^{a}$ for some integer $a$.
\end{enumerate}
\end{lemma}
The first property is Lemma 3 in \cite{D12} and the second property easily follows from Lemma 2 in \cite{F02}.
\begin{definition}
A primitive prime divisor $p$ for a term $a_{n}$ in an integer sequence $(a_{i})_{i=1}^{\infty}$ is a prime dividing $a_{n}$ but not dividing $a_{i}$ for all positive integers $i < n$.
\end{definition}
\begin{lemma}[\textbf{Zsigmondy's theorem, \cite{EW05}}]
\label{Zsigemondy}
Let $b \geq 2$ be an integer. Any term in the sequence $(\Phi_{d}(b))_{d=1}^{\infty}$ has a primitive prime divisor, except for the cases
\begin{itemize}
    \item $b = 2$ and $d = 1$,
    \item $b=2^{a}-1$ and $d=2$,
    \item $b = 2$ and $d = 6$.
\end{itemize}
\end{lemma}
We will show that $\Phi_{d}(b)$ typically has a prime divisor $>d$.
\begin{lemma} \label{largePrimitive}
Let $b \geq 2$ and $d$ be positive integers such that $(b,d) \neq (2,1)$, $(b,d) \neq (2,6)$ and $(b,d) \neq (2^{a}-1,2)$ for all positive integers $a$. Then there is a prime, $p >d$, satisfying $p \mid \Phi_{d}(b)$.
\end{lemma}
We give the proof of Lemma \ref{largePrimitive} in the next section. The proof in fact implies more, namely that, for $(b, d)$ as in Lemma \ref{largePrimitive}, we have that every primitive prime divisor of $\Phi_{d}(b)$ for the sequence $(\Phi_{d}(b))_{d=1}^{\infty}$ is greater than $d$.

\section{Proofs of Lemmas 3.2 and 3.7}
Before we prove Theorem \ref{main1} we will settle Lemma \ref{rootimplication} and Lemma \ref{largePrimitive}.
\begin{proof}[Proof of Lemma \ref{rootimplication}]
Since the polynomials in question are monic, we can do division in $\mathbb{Z}[x]$ and obtain
\begin{equation}
 g(x) = q(x)h(x)+r(x) \label{division}
\end{equation}
with $q(x),r(x) \in \mathbb{Z}[x]$ and $\text{deg}(r(x))<\text{deg}(h(x))$. We can pick a prime $p$ greater than all coefficients in $r(x)$ such that $h(x) \mod p$ splits completely in $\mathbb{F}_{p}[x]$ and such that (\ref{rootImplication}) is satisfied. Evaluating (\ref{division}) in the $\text{deg}(h(x))$ many distinct roots of $h(x) \mod p$ implies that $r(x) \mod p$ must be the zero polynomial. But as $p$ was chosen large this implies that $r(x) = 0$ holds in $\mathbb{Z}[x]$. Thus, $h(x)$ divides $g(x)$ in $\mathbb{Z}[x]$ as claimed.
\end{proof}
\begin{proof}[Proof of Lemma \ref{largePrimitive}] For $d = 1$ the claim is obviously fulfilled. Assume therefore that $d>1$.
By Lemma \ref{Zsigemondy} we may pick a prime $p$ such that $p \mid \Phi_{d}(b)$ and $p \nmid \Phi_{i}(b)$ for $1 \leq i \leq d-1$. We claim that $p > d$. Assume for the sake of a contradiction that $p \leq d$.

As $\Phi_{d}(0) = 1$ we must have $(p,b) = 1$, and therefore by Fermat's Little Theorem
\[
p \mid (b^{p-1}-1) = \prod_{k \mid (p-1)}\Phi_{k}(b),
\]
implying that $p \mid \Phi_{k}(b)$ for some $k \leq p-1<d$, contradicting the fact that $p$ is a primitive prime divisor.
\end{proof}

\section{Proof of the main results}

\begin{proof}[Proof of Theorem \ref{main1}]
Let first $f(x) \in \mathbb{Z}[x]$ be such that there is a constant $N=N(f)$ for which (\ref{divisorprop}) is satisfied.

As $f(x)$ is monic we can factorize it as, say,
\[f(x)=x^{a_{0}}p_{1}(x)^{a_{1}}p_{2}(x)^{a_{2}}\cdots p_{m}(x)^{a_{m}} \]
where the $p_{i}(x)$ are distinct irreducible monic polynomials in $\mathbb{Z}[x]$ different from $x$. Let
\[g(x): = p_{1}(x)p_{2}(x)\cdots p_{m}(x):=x^{n}+a_{n-1}x^{n-1}+\cdots +a_{1}x+a_{0}.\]
It is easy to see that $g(x)$ must then also be monic, have a nonzero discriminant and satisfy the weaker condition
\begin{equation}
	\text{rad}(g(p)) \mid g(p^{p}) ~\text{for all~}p \geq \max\{ N, |a_{0}|\}. \label{weakdivisorprop}
\end{equation}
If $g(x) = 1$ we are done, so assume that $\deg(g(x)) \geq 1$.
\begin{lemma} \label{boundedOrderLemma}
There is a constant $C:=C(g(x))$ only dependent on $g(x)$ such that there are infinitely many primes $q$ for which $g(x) \mod q$ splits completely and such that any root of $g(x) \mod q$ has order at most $C$ modulo $q$.
\end{lemma}
\begin{proof}
Consider primes $q$ satisfying $q > \max\{N,|a_{0}|\}$ and for which there is an integer $r \in \mathbb{Z}$ such that $q \mid g(r)$. Note that as $a_{0} \neq 0$ we must have $(r,q) =1 $, and thus by Dirichlet's Theorem for Primes in Arithmetic Progressions there are infinitely many primes $p>r$ such that $p \equiv r \pmod{q}$. Since $p>q>N$, we obtain from (\ref{weakdivisorprop}) that
\[q \mid \text{rad}(g(p)) \mid g(p^{p}).\]
Thus, $q \mid g(r^{p})$ for any prime $p>r$ such that $p \equiv r \pmod{q}$. However, combining the Chinese Remainder Theorem and Dirichlet's Theorem for Primes in Arithmetic Progressions we are free to choose any reduced residue for $p \pmod{q-1}$. Thus, for any positive integer $l$ coprime to $q-1$ we must have $q \mid g(r^{l}) $. But as the reduction of $g(x)$ modulo $q$ is nonzero there can be at most $\text{deg}(g(x))$ distinct values in the sequence $(r^{l} \mod q)_{(l,q-1)=1}$. 

By Theorem \ref{specialChebotarev}, there is a density $\delta \in (0,1]$ of primes $q$ for which $g(x)$ splits completely modulo $q$. Choose $T$ so large that the $T$ first primes $p_{1},p_{2},\dots ,p_{T}$ satisfy
\[
(p_{1}-1)(p_{2}-1)\cdots (p_{T}-1) \geq 2\delta^{-1},
\]
and put $C=p_{T}^{\text{deg}(g(x))+2}$. By the Prime Number Theorem for Primes in Arithmetic Progressions there are infinitely many primes $q$ such that $g(x)$ splits completely modulo $q$ and such that $q \not \equiv 1 \pmod{ p_{1}p_{2}\cdots p_{T}}$. By the pigeonhole principle there must be an $1 \leq i \leq T$ such that there are infinitely many such primes $q \not \equiv 1 \pmod{p_{i}}$. Let
\[
Q_{i}: = \{ q~\text{prime}: q \not\equiv 1 \pmod {p_{i}},~\text{and}~g(x)~\text{splits completely modulo}~q\}
\]
which then is an infinite set.
We claim that with our choice of $C$ the elements of $Q_{i}$ satisfy that the order of any root of $g(x)$ modulo $q$ is at most $C$.

To check this, let $q \in Q_{i}$ and let $r$ be a root of $g(x)$ modulo $q$. If the order of $r$ modulo $q$ were greater than $C$ we would have
\[
\{ r^{p_{i}} \mod q,r^{p_{i}^{2}} \mod q,\dots, r^{p_{i}^{\text{deg}(g(x))+1}} \mod q \} \subseteq (r^{l} \mod q )_{(l,q-1)=1}
\]  
which would be a contradiction as the elements on the left-hand side would be pairwise distinct.
\end{proof}
\begin{lemma}
There is an integer $M$ such that $g(x) \mid (x^{M}-1)$.
\end{lemma}
\begin{proof}
Let $M$ be the least common multiple of all positive integers less than or equal to $C$. By Lemma \ref{boundedOrderLemma} there are then infinitely many primes $q$ for which $g(x)$ modulo $q$ splits completely and any root, $r$, of $g(x) \mod q$ satisfies $r^{M}-1 \equiv 0 \pmod{q}$. Thus, we are in a situation where we can apply Lemma \ref{rootimplication} to obtain that $g(x) \mid (x^{M}-1)$ in $\mathbb{Z}[x]$.
\end{proof}
Using the identity 
\[x^{M}-1 = \prod_{d \mid M}\Phi_{d}(x) \]
we now easily obtain that $g(x)$ is the product of certain cyclotomic polynomials and we thereby get that $f(x)$ has the form of (\ref{primeFac}), as wanted.

Now to see that all polynomials of the form (\ref{primeFac}) satisfy (\ref{divisorprop}) it clearly suffices to check that for every $i=1,2,\dots, m$,
\[ \Phi_{d_{i}}(p) \mid \Phi_{d_{i}}(p^{p}), \text{ for all primes }p \geq \max\{ d_{1},d_{2},\dots, d_{m}  \} +1.
\]
This, however, follows directly from Lemma \ref{sufLemma}, concluding the proof of Theorem \ref{main1}.
\end{proof}

\begin{proof}[Proof of Theorem \ref{oneRoot}]
We start by showing that any monic polynomial \newline $f(x) \in \mathbb{Z}[x]$ that satisfies (\ref{divisorprop}) with $N=2$, and that is not a power of $x$, must be divisible by $x-1$.

Let
\[
 f(x) = x^{e_{0}}\Phi_{d_{1}}(x)^{e_{1}} \Phi_{d_{2}}(x)^{e_{2}} \cdots \Phi_{d_{m}}(x)^{e_{m}}
\]
satisfy (\ref{divisorprop}) with $N=2$ and with $f(x)$ not a power of $x$. As $(\Phi_{d}(n),n) = 1$ for all $d$ and $n$ we must also have that $h(x):= f(x)/x^{e_{0}}$ satisfies (\ref{divisorprop}) with $N=2$.

Let $\mathcal{A} = \{ d_{1},d_{2},\dots ,d_{m}  \}$. We then have to show that $1 \in \mathcal{A}$. Let
\[M:= \text{lcm}(d_{1},d_{2},\dots , d_{m}) := p_{1}^{a_{1}}p_{2}^{a_{2}}\cdots p_{k}^{a_{k}} \]
where
\[
p_{1} < p_{2}< \cdots < p_{k}.
\]
\begin{lemma} \label{reductionLemma}
Let $d \in \mathcal{A}$ for which $p_{k} \mid d$. If $p_{i} \mid d$ and $(p_{i},d) \neq (2,6)$ then also $d/p_{i} \in \mathcal{A}$.
\end{lemma}
\begin{proof}
According to Lemma \ref{largePrimitive} $\Phi_{d}(p_{i})$ must have a prime factor, $p$, satisfying $p >d \geq p_{k}$. As (\ref{divisorprop}), in particular, holds for $p_{i}$ there is an index $j$ such that $p \mid \Phi_{d_{j}}(p_{i}^{p_{i}})$. By Lemma \ref{sufLemma} (1) we have $p \mid \Phi_{d/p_{i}}(p_{i}^{p_{i}})$ and Lemma \ref{sufLemma} (2) then gives that $d_{j}/(d/p_{i})$ is a power of $p$. But as $p>p_{k}$ this means that $d_{j} = d/p_{i}$, proving the lemma.
\end{proof}

Starting with an element $d \in \mathcal{A}$ that is divisible by $p_{k}$ we may repeatedly apply Lemma \ref{reductionLemma} to peel off prime factors of $d$, in each step replacing $d$ by $d/p_{i}$, where $p_{i}$ is the least prime factor of $d$. Continuing this process as long as possible we obtain either that $1 \in \mathcal{A}$ or that $6 \in \mathcal{A}$ and $p_{k} = 3$.

In the first case we are done, so assume that $6 \in \mathcal{A}$ and $p_{k} = 3$. As $\Phi_{6}(2) = 3$ there must be some $d_{j} \in \mathcal{A}$ with $3 \mid \Phi_{d_{j}}(2^{2})$. Since $3 \mid \Phi_{1}(4)$, $d_{j}$ must by Lemma \ref{sufLemma} (2) be a power of 3. Putting $d = d_{j}: = 3^{a}$ we may again repeatedly apply Lemma \ref{reductionLemma}, ultimately leading to $1 \in \mathcal{A}$.

For the second statement let $d_{2} = p_{2},~d_{3} = p_{3}, \dots, d_{m} = p_{m}$ be distinct primes. The first statement in Lemma \ref{sufLemma} implies that $\Phi_{p_{j}}(p) \mid \Phi_{p_{j}}(p^{p})$ for any prime $p \neq p_{j}$. Therefore we only need to worry about $p = p_{j}$ for some $j$. However, $\Phi_{p_{j}}(p_{j})\Phi_{1}(p_{j}) = p_{j}^{p_{j}}-1 = \Phi_{1}(p_{j}^{p_{j}})$, which settles the case $p=p_{j}$.
\end{proof}

\subsection*{Acknowledgements}
The author thanks the anonymous referee for providing valuable insights and giving detailed comments.
The author also thanks his supervisor, Christian Elsholtz, for commenting on a preprint of the paper.  The author acknowledges the support of the Austrian Science Fund (FWF): W1230.

\end{document}